\documentclass[12pt,a4paper]{article}
\usepackage[margin=2.5cm]{geometry}
\usepackage{amsmath,amssymb,amsthm}
\usepackage{booktabs}
\usepackage{graphicx}
\usepackage{xcolor}
\usepackage{tikz}
\usetikzlibrary{decorations.pathreplacing}
\usepackage[hidelinks]{hyperref}
\newtheorem{theorem}{Theorem}
\newtheorem{proposition}[theorem]{Proposition}
\newtheorem{lemma}[theorem]{Lemma}
\newtheorem{corollary}[theorem]{Corollary}
\newcommand{\PP}{\mathbb{P}}
\newcommand{\EE}{\mathbb{E}}
\newcommand{\ind}{\mathbf 1}
\newcommand{\Cor}{C_c^{\rightarrow}}
\newcommand{\thor}{\theta^{\rightarrow}}
\newcommand{\step}[2]{\medskip\textbf{Step #1 (#2).}}

\begin{document}

\begin{center}
{\LARGE Explicit upper bounds on the threshold of\\[4pt] one-dimensional long-range percolation\par}
\vspace{1.1em}
{\large Guy Amit\par}
\vspace{0.5em}
Department of Natural Sciences, The Open University of Israel, Ra'anana, Israel\par
\vspace{0.3em}
E-mail: \texttt{guyam@openu.ac.il}
\end{center}

\vspace{0.6em}
\noindent\textbf{Abstract.}
We study bond percolation on the one-dimensional lattice in which two sites at distance $r$ are connected with
probability $p_r=C/r^{1+\sigma}$, where $0<\sigma<1$. An infinite cluster exists if $C$ exceeds a critical value
$C_c(\sigma)$, which is not known exactly: the available results are lower bounds, such as Schulman's
$C_c\ge1/[2\zeta(1+\sigma)]$, and numerical estimates. Here we prove explicit upper bounds. We apply a second-moment
argument to a family of random monotone paths whose steps have a heavy-tailed (Sibuya) distribution. Two such paths
meet only a finite number of times, and their overlap is computed exactly by renewal theory. The result is an
explicit bound, which gives $C_c\le\sigma+O(\sigma^2)$ as $\sigma\to0$ and is below $1$ for $\sigma<0.667$.
Since the paths are monotone, the bound holds also for oriented percolation, and it determines the oriented
threshold asymptotically: $\Cor\,\zeta(1+\sigma)\to1$ as $\sigma\to0$. A coarse-grained version of the argument,
with blocks of sites in place of sites, gives an explicit bound $C_c<1$ for every $\sigma<1$. We compare the bounds
with the known numerical values of $C_c$.

\vspace{0.4em}

\section{Introduction}
\label{sec:intro}

Long-range (LR) percolation on a one-dimensional lattice was introduced by Zhang, Pu and Li~\cite{zhang_long-range_1983} and
by Schulman~\cite{schulman_long_1983}. Schulman was motivated by the Ising chain with LR interactions~\cite{dyson_existence_1969,defenu_long-range_2023}. In LR percolation, every pair of sites at distance $r$ is connected by a bond with probability
\begin{equation}\label{eq:pr}
p_r=\frac{C}{r^{1+\sigma}},
\end{equation}
independently of all other pairs. The model has two parameters: the exponent $\sigma>0$, and the amplitude
$0<C\le1$, which is the probability of a bond between two sites at unit distance. If the bonds have a finite range, a
one-dimensional chain has no transition at a nontrivial value of the parameter: an infinite cluster exists only if
nearest neighbours are connected with probability one. The LR bonds can change this. For $\sigma>1$ the long bonds
are too rare: with positive probability no bond passes over a given point, and there is still no infinite cluster
unless nearest neighbours are connected with probability one ($C_c=1$)~\cite{schulman_long_1983}. For $0<\sigma<1$ there is a
transition at a critical value $0<C_c(\sigma)<1$, as proved by Newman and Schulman~\cite{newman_one_1986} with a
multi-scale renormalization argument. If $\sigma=1$ there is no infinite cluster for any
$C<1$~\cite{aizenman_discontinuity_1986}. A transition at $\sigma=1$ occurs only in a wider family of models, in which the long
bonds have the probability $p_r\approx\beta/r^2$ with an amplitude $\beta>1$, while the probability of the
nearest-neighbour bonds is a separate parameter; this transition is
discontinuous~\cite{aizenman_discontinuity_1986,duminil-copin_long-range_2024}. More recently, the critical behaviour for $0<\sigma<1$ was studied
rigorously. There is no infinite cluster at the critical point~\cite{berger_transience_2002}, the cluster-size distribution at
criticality obeys power-law bounds~\cite{hutchcroft_power-law_2021,baumler_isoperimetric_2024}, and mean-field behaviour was proved for
$\sigma<1/3$~\cite{hutchcroft2025critical}.

Despite the simplicity of the model, the value of $C_c(\sigma)$ is not known. Schulman~\cite{schulman_long_1983} proved,
by a comparison with a Bethe lattice, the lower bound
\begin{equation}\label{eq:schulman}
C_c(\sigma)\ge C_S(\sigma)=\frac{1}{2\zeta(1+\sigma)},
\end{equation}
where $\zeta(s)=\sum_{n=1}^\infty n^{-s}$ is the Riemann zeta function. Spouge~\cite{spouge_long-range_1984} strengthened this bound slightly, by a
comparison with a branching process that follows two generations at a time. On the numerical side, Rego
\textit{et al}~\cite{rego_crossover_1999} found that the threshold vanishes linearly as $\sigma\to0$, and Gori
\textit{et al}~\cite{gori_one-dimensional_2017} determined $C_c(\sigma)$ with high precision and observed that it approaches
$C_S(\sigma)$ as $\sigma\to0$. Upper bounds have received less attention: the proof of~\cite{newman_one_1986} that
$C_c<1$ does not give a numerical value, and we have not found an explicit upper bound on $C_c(\sigma)$ in the
literature.

The value of $C_c$ is also of practical interest. LR percolation describes spreading processes with
immunization in which the infection is transmitted over distances with a power-law distribution, and $C_c$ is
then the epidemic threshold~\cite{grassberger_sir_2013,janssen_levy-flight_1999,linder_long-range_2008}. It also marks the onset of the regime in which a
network with LR links has an infinite cluster with very short paths, whose length grows only as a power of the
logarithm of the distance~\cite{biskup_scaling_2004}. Finally, it is the reference value for more complex LR networks, such
as fractal and interdependent networks~\cite{amit_percolation_2023,sharabi2026percolation}.

Here we prove explicit upper bounds on $C_c(\sigma)$. Our results are the following.
\begin{itemize}
\item A lower bound on the percolation probability, for every choice of a step distribution of a random monotone
path (theorem~\ref{thm:main}).
\item An explicit choice of the step distribution gives an upper bound $C^U(\sigma)$ on $C_c$, with a closed form
that behaves as $\sigma+O(\sigma^2)$ for small $\sigma$ (corollary~\ref{cor:bound}). Together with
\eqref{eq:schulman}, this shows rigorously that $C_c$ vanishes linearly with $\sigma$, as found numerically
in~\cite{rego_crossover_1999}: in the limit $\sigma\to0$
the ratio $C_c/\sigma$ lies between $1/2$ and $1$.
\item The same bound holds for \emph{oriented} percolation, in which bonds are used in one direction only. For this
model the bound is asymptotically sharp: the threshold of oriented percolation, which we denote by $\Cor$,
satisfies $\Cor\,\zeta(1+\sigma)\to1$ as
$\sigma\to0$, and the upper and lower bounds differ by $0.5\%$ at $\sigma=0.1$.
\item A coarse-grained version, with blocks of $L$ sites, gives an explicit bound $C_c<1$ for all $\sigma<1$
(theorem~\ref{thm:block} and corollary~\ref{cor:less1}).
\end{itemize}

For example, for $0<\sigma<\frac12$ an infinite cluster exists whenever
\[
C>\sigma^2\,\zeta(1+\sigma)\,\frac{\Gamma(1-2\sigma)}{\Gamma(1-\sigma)^2}.
\]
For $\sigma=0.2$ the right-hand side is $0.246$, and the optimized bound of corollary~\ref{cor:bound} is $0.1851$.
For comparison, at this $\sigma$ there is no infinite cluster for $C<0.089$ by \eqref{eq:schulman}, and the
numerical threshold is $0.093$~\cite{gori_one-dimensional_2017}.

Our method is the second-moment method~\cite{alon_probabilistic_2016}, applied to a weighted count of open
paths~\cite{lyons_probability_2016}. One chooses a random path, independently of the bonds, and shows that the second moment
of the number of open paths is not much larger than the square of its mean. This implies that open paths of any
length exist with positive probability. The method was used by Cox and Durrett~\cite{cox_oriented_1983} to bound the
threshold of oriented nearest-neighbour percolation on the hypercubic lattice $\mathbb{Z}^d$ with $d\ge4$, and by
Lyons~\cite{lyons_random_1990} to find the threshold of percolation on general trees. Benjamini, Pemantle and
Peres~\cite{benjamini_unpredictable_1998} applied it to oriented nearest-neighbour percolation on $\mathbb{Z}^3$, and
B\"aumler~\cite{baumler2024truncation} recently applied it to percolation on $\mathbb{Z}^d$, $d\ge3$, with bonds of
unbounded length. In all these works the key property is that two independent random paths rarely intersect.

In one dimension this property fails for paths with short steps. Consider an infinite monotone path whose step
lengths have a finite mean $\mu$ and no common factor. By the renewal theorem~\cite{feller_introduction_1968}, it visits
a distant site with probability approximately equal to $1/\mu$. Two
independent paths of this kind therefore both visit a distant site with probability close to $1/\mu^2$, and they
have infinitely many common sites. If instead the step lengths have a LR distribution $\pi$, in which the probability
$\pi(r)$ of a step of length $r$ decays as $r^{-1-\alpha}$ with $0<\alpha<1$, the path visits a distant site $x$ with a probability of order
$x^{\alpha-1}$~\cite{garsia_discrete_1962,doney_one-sided_1997}, which tends to zero. Two independent paths both visit $x$ with a
probability of order $x^{2\alpha-2}$, and the mean number of their common sites, which is the sum of these
probabilities over $x$, is finite if and only if $\alpha<1/2$. Thus, for $\alpha<1/2$, two independent infinite
paths have only a finite number of common sites. The same fact explains why weak disorder does not
change the critical behaviour of certain pinning models~\cite{alexander_effect_2008,toninelli_replica-coupling_2008}.

There is a competing requirement. A step of length $r$ uses a bond that is open only with the small probability
$p_r=C/r^{1+\sigma}$. If two paths share such a bond, the probability that both paths are open is not the product
of their separate probabilities, but is larger by the factor $1/p_r$, since the shared bond has to be open only
once. The mean of this factor is finite only if the steps are not too long, namely if
$\alpha>\sigma/2$ (section~\ref{sec:results}). The two requirements, $\sigma/2<\alpha<1/2$, can be met together if
and only if $\sigma<1$, which is exactly the range in which the model has a transition.

The paper is organized as follows. In section~\ref{sec:results} we define the model and state the results. Theorem~\ref{thm:main}
is proved in section~\ref{sec:proof}, the explicit bound in section~\ref{sec:sibuya}, and the bound with blocks in
section~\ref{sec:blocks}. The numerical evaluation of the bounds is described in section~\ref{sec:numerics}, and
the results are discussed in section~\ref{sec:discussion}.

\section{Model and results}
\label{sec:results}

In this section we define the model and state all the results. The proofs are given in
sections~\ref{sec:proof}--\ref{sec:blocks}.

\subsection{The model}

The sites are the integers $\mathbb{Z}$. Every pair of sites at distance $r\ge1$ is connected by a bond with the
probability $p_r$ of \eqref{eq:pr}, independently of all other pairs, where $0<\sigma<1$ and $0<C\le1$. A bond that
is present is called \emph{open}. A \emph{path} is a finite or infinite sequence of sites $x_0,x_1,x_2,\dots$. The move from
$x_{k-1}$ to $x_k$ is the $k$th \emph{step} of the path, and $|x_k-x_{k-1}|$ is the length of the step, so that a
path $x_0,\dots,x_n$ has $n$ steps. The path is open if the bonds $\{x_{k-1},x_k\}$ of all its steps, $k=1,2,\dots$, are open. The
\emph{cluster} of a site is the
set of sites that can be reached from it by open paths. The percolation probability is
\[
\theta(C)=\PP(\text{the cluster of site }0\text{ is infinite}).
\]
The function $\theta(C)$ is nondecreasing, since the model with a larger $C$ can be obtained from the model with a
smaller $C$ by opening additional bonds. The threshold $C_c$ is the value of $C$ that separates the two cases:
$\theta(C)=0$ for $C<C_c$ and $\theta(C)>0$ for $C>C_c$.

We also consider the \emph{oriented} version of the model. The bonds and their probabilities are the same, but only
monotone paths are allowed. A \emph{monotone path} is a path with $x_0<x_1<x_2<\cdots$, so that every bond is
traversed from its left end to its right end. The \emph{forward cluster} of site $0$ is the set of sites that can
be reached from $0$ by open monotone paths. We denote
\[
\thor(C)=\PP(\text{the forward cluster of site }0\text{ is infinite}),
\]
and its threshold, defined in the same way, by $\Cor$.
Every monotone path is a path, so the forward cluster of $0$ is contained in its cluster. Hence $\thor\le\theta$ and
\begin{equation}\label{eq:order}
C_c\le\Cor\le1 .
\end{equation}
Oriented percolation in this model was considered in~\cite{newman_one_1986}, and for $\sigma=1$ in~\cite{marchetti_oriented_2010}.

The lower bound \eqref{eq:schulman} has an oriented analog. Both follow from a count of open paths, which we recall
for completeness.

\begin{proposition}[first-moment bounds]\label{prop:lower}
For $0<\sigma<1$,
\[
C_c(\sigma)\ge\frac{1}{2\zeta(1+\sigma)},\qquad \Cor(\sigma)\ge\frac{1}{\zeta(1+\sigma)} .
\]
\end{proposition}

\begin{proof}
Let $N_n$ be the number of open monotone paths of $n$ steps from $0$. Such a path is determined by its step lengths
$r_1,\dots,r_n\ge1$. It uses $n$ different bonds, which are open independently of each other, so it is open with
probability $p_{r_1}\cdots p_{r_n}$. The mean of $N_n$ is the sum of these probabilities over all the paths, and
the sum of the products factorizes:
\[
\EE N_n=\sum_{r_1,\dots,r_n\ge1}p_{r_1}\cdots p_{r_n}=\Bigl(\sum_{r\ge1}p_r\Bigr)^n=[C\zeta(1+\sigma)]^n .
\]
Next we show that if the forward cluster of $0$ is infinite, then $N_n\ge1$ for every $n$, and therefore
$\thor(C)\le\PP(N_n\ge1)$. The expected number of open bonds that have a given site as one of their two ends (the
\emph{mean degree} of a site) is $\langle k\rangle=2\sum_rp_r=2C\zeta(1+\sigma)<\infty$, so with probability one every
site has a finite number of such bonds. Let $A_k$ be the set of sites that can be reached from $0$ by an open
monotone path of exactly $k$ steps. Every site of $A_{k+1}$ is connected by an open bond to a site of $A_k$, namely
to the site that precedes it on the path. We show by induction that every $A_k$ is finite. First, $A_0=\{0\}$ is
finite. Next, if $A_k$ is finite for some $k$, then $A_{k+1}$ is finite too, since every site of $A_{k+1}$ is the
other end of one of the finitely many open bonds of the finitely many sites of $A_k$. Hence $A_k$ is finite for
every $k\ge0$. The
forward cluster is the union of all the $A_k$. If it is infinite, it is not contained in the finite set
$A_0\cup\dots\cup A_{n-1}$, so there is an open monotone path of some $k\ge n$ steps, and its first $n$ steps are an
open monotone path of $n$ steps.

Finally, $\PP(N_n\ge1)\le\EE N_n$ by Markov's inequality. Together,
\[
\thor(C)\;\le\;\PP(N_n\ge1)\;\le\;\EE N_n=[C\zeta(1+\sigma)]^n\qquad\text{for every }n .
\]
If $C\zeta(1+\sigma)<1$ the right-hand side tends to zero as $n\to\infty$, and therefore $\thor(C)=0$.

For the unoriented model we count self-avoiding paths, whose steps can be to the right or to the left. A step of
length $r$, in either direction, is open with probability $p_r$, and dropping the restriction that the path is
self-avoiding can only increase the sum. The expected number of open self-avoiding paths of $n$ steps from $0$ is
therefore at most $\bigl(2\sum_{r\ge1}p_r\bigr)^n=[2C\zeta(1+\sigma)]^n$, and the same argument gives $\theta(C)=0$
if $2C\zeta(1+\sigma)<1$.
\end{proof}

Both statements of proposition~\ref{prop:lower} are lower bounds: they tell us when there is no percolation. In the
opposite direction, \eqref{eq:order} shows that every upper bound on $\Cor$ is also an upper bound on $C_c$. This is
how we obtain our upper bounds: theorem~\ref{thm:main} below gives a sufficient condition for oriented percolation.

\subsection{The main result}

Let $\pi$ be a probability distribution on the step lengths $\{1,2,3,\dots\}$. Thus $\pi(r)\ge0$ for every $r$, and
$\sum_{r\ge1}\pi(r)=1$. A \emph{random monotone path} is an infinite sequence of sites
\[
X_0<X_1<X_2<\cdots
\]
that starts at site $X_0=0$, and whose steps $X_k-X_{k-1}$ are independent, each with distribution $\pi$. The path is drawn independently of the
bonds, and it is not required to be open. The distribution $\pi$ is not a property of the model. It is a tool of
the proof, which we are free to choose, and it should not be confused with the bond probabilities $p_r$, which do
not sum to one.

The bound of theorem~\ref{thm:main} depends on $\pi$ through two quantities. The first is the probability that the
path visits a given site $x\ge0$,
\[
u(x)=\PP\bigl(x\in\{X_0,X_1,X_2,\dots\}\bigr),\qquad u(0)=1,
\]
which is known as the \emph{renewal sequence} of $\pi$~\cite{feller_introduction_1968}. The second is the sum
\begin{equation}\label{eq:W}
W(\pi)=\sum_{x\ge0}u(x)^2 .
\end{equation}
For two independent paths, $u(x)^2$ is the probability that both visit $x$. Hence $W(\pi)$ is the mean number of
sites that two independent infinite paths have in common. Both paths start at site $0$, so site $0$ is one of them and $W\ge1$.

\begin{theorem}\label{thm:main}
For every $\pi$ with $W(\pi)<\infty$, and every $0<C\le1$,
\begin{equation}\label{eq:main}
\theta(C)\;\ge\;\thor(C)\;\ge\;1-W(\pi)\sum_{r\ge1}\pi(r)^2\Bigl(\frac1{p_r}-1\Bigr).
\end{equation}
In particular $C_c\le\Cor\le C$ whenever the right-hand side is positive.
\end{theorem}

The right-hand side of \eqref{eq:main} can be understood as follows; the precise argument is given in
section~\ref{sec:proof}. The proof introduces a random variable $Z$, which counts the open monotone paths from site $0$,
with suitable weights, such that $\EE Z=1$ and $\PP(Z>0)\ge1/\EE Z^2$. The second moment $\EE Z^2$ turns out to be an average over
pairs of independent random paths. A pair of paths that have no bond in common contributes $1$ to this average,
and every bond of length $r$ that the two paths share multiplies the contribution of the pair by $1/p_r\ge1$.
Shared bonds therefore increase $\EE Z^2$ and weaken the bound, and we refer to the excess $1/p_r-1$ as the
\emph{cost} of a shared bond. At a site that both paths visit, they take the same step $r$ with probability
$\pi(r)^2$. The sum in \eqref{eq:main} is therefore the mean cost per common site, and $W$ is the mean number of
common sites. The bound shows that percolation occurs when their product is smaller than $1$, and the smaller
the product, the larger the lower bound on $\theta$.

For the bound to be useful, both factors on the right-hand side of \eqref{eq:main} must be finite. Suppose that
the steps have a power-law tail, $\pi(r)\sim r^{-1-\alpha}$, so that a smaller $\alpha$ means longer steps. The
mean cost is finite if and only if the sum $\sum_r\pi(r)^2/p_r\propto\sum_rr^{-2-2\alpha}\,r^{1+\sigma}$ converges,
that is, if and only if $\alpha>\sigma/2$: the steps must not be too long. On the other hand, steps with a
power-law tail give $u(x)\sim x^{\alpha-1}$~\cite{garsia_discrete_1962,doney_one-sided_1997}. The reason is that the largest of $n$ steps
is of order $n^{1/\alpha}$, so that a path covers a distance of order $n^{1/\alpha}$ in $n$ steps, and it visits about
$x^{\alpha}$ of the first $x$ sites, a fraction $x^{\alpha-1}$ of them. (For the steps that we use,
lemma~\ref{lem:sibuya} gives $u(x)$ exactly.) Hence $W=\sum_xu(x)^2$ is finite if and only if $\alpha<1/2$: the
steps must be long enough for two independent paths to have only a finite number of common sites. Both conditions hold if and only if
$\alpha$ lies in the window
\begin{equation}\label{eq:window}
\frac{\sigma}{2}<\alpha<\frac12 ,
\end{equation}
which is not empty if and only if $\sigma<1$. As $\sigma$ grows the window becomes narrower, and as $\sigma\to1$
it shrinks to the single point $\alpha=\frac12$. Every $\alpha$ in a narrow window is close to one of its ends, where
one of the two factors is large: $W$ diverges as $\alpha\to\frac12$, and the mean cost diverges as
$\alpha\to\sigma/2$. This is why the upper bounds on $C_c$ that we obtain become weaker as $\sigma$ grows
(figure~\ref{fig:bounds}): the bound of corollary~\ref{cor:bound} exceeds $1$ for $\sigma>0.667$, and the bound of
theorem~\ref{thm:block}, with blocks, approaches the trivial bound $C_c\le1$ as $\sigma\to1$.

\begin{figure}[t]
\centering
\includegraphics[width=0.8\textwidth]{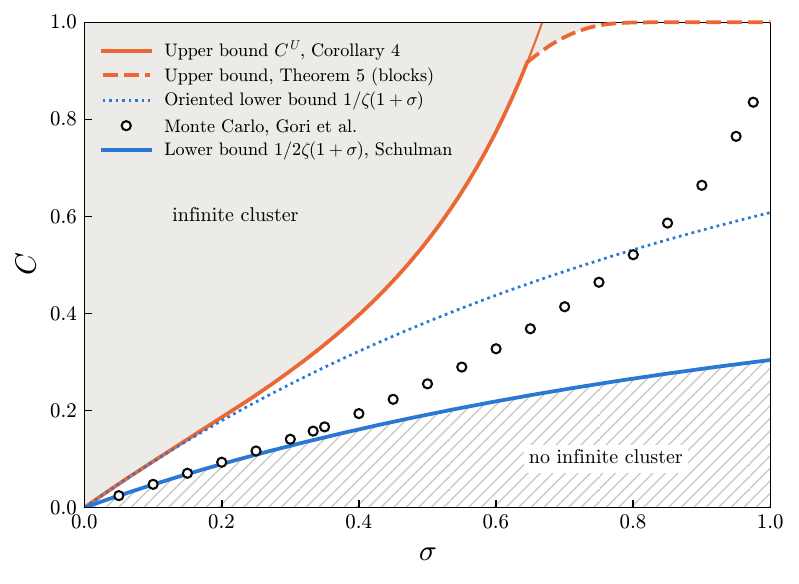}
\caption{Bounds on the thresholds as functions of $\sigma$. Thick blue line: the lower bound \eqref{eq:schulman}
of Schulman~\cite{schulman_long_1983} on $C_c$. Circles: the Monte Carlo values of $C_c$ of Gori
\textit{et al}~\cite{gori_one-dimensional_2017}; their error bars are smaller than the symbols. Orange lines: the upper bounds of
this work on $C_c$, namely $C^U$ of corollary~\ref{cor:bound} (solid) and the bound of theorem~\ref{thm:block} with
blocks (dashed), which is the better one for $\sigma>0.645$. The dashed line is below $1$ for every $\sigma<1$, but
for $\sigma>0.8$ it is closer to $1$ than the line width. An infinite cluster exists in the
shaded region and does not exist in the striped region, and the threshold $C_c$ lies between them. Dotted
blue line: the lower bound $1/\zeta(1+\sigma)$ on the oriented threshold $\Cor$ (proposition~\ref{prop:lower}).
The upper bound on $\Cor$ is the orange line $C^U$, which is drawn thin beyond $\sigma=0.645$ and reaches $C=1$ at
$\sigma\approx0.667$. Thus $\Cor$ lies between the dotted blue line and the orange line $C^U$.}
\label{fig:bounds}
\end{figure}

\subsection{An explicit bound}

To turn theorem~\ref{thm:main} into an explicit bound on the threshold we need a step distribution with a
power-law tail for which $W$ can be calculated in closed form. A convenient choice is the Sibuya distribution~\cite{sibuya_generalized_1979}
$\pi_\alpha$, with $0<\alpha<1$, defined by its generating function
\begin{equation}\label{eq:gf}
\sum_{r\ge1}\pi_\alpha(r)\,z^r=1-(1-z)^\alpha .
\end{equation}

\begin{lemma}[Sibuya steps]\label{lem:sibuya}
For $0<\alpha<1$:
\begin{itemize}
\item[(S1)] $\displaystyle\pi_\alpha(r)=\frac{\alpha}{r}\prod_{k=1}^{r-1}\Bigl(1-\frac{\alpha}{k}\Bigr)>0$, and
$\sum_{r\ge1}\pi_\alpha(r)=1$. (For $r=1$ the product is empty and equals $1$, so $\pi_\alpha(1)=\alpha$.)
\item[(S2)] $\pi_\alpha(r)\le\alpha\,r^{-1-\alpha}$.
\item[(S3)] The renewal sequence is $\displaystyle u(x)=\frac{\Gamma(x+\alpha)}{\Gamma(\alpha)\,x!}$.
\item[(S4)] For $0<\alpha<\frac12$,
\begin{equation}\label{eq:Wsib}
W(\alpha):=W(\pi_\alpha)=\frac{\Gamma(1-2\alpha)}{\Gamma(1-\alpha)^2},\qquad
\sum_{r\ge1}\pi_\alpha(r)^2=\frac{\Gamma(1+2\alpha)}{\Gamma(1+\alpha)^2}-1,
\end{equation}
and $W(\alpha)=\infty$ for $\alpha\ge\frac12$.
\end{itemize}
\end{lemma}

\begin{corollary}\label{cor:bound}
For $0<\sigma<1$,
\begin{equation}\label{eq:CU}
C_c(\sigma)\;\le\;\Cor(\sigma)\;\le\;C^U(\sigma):=\min_{\sigma/2<\alpha<1/2}\;
\frac{\displaystyle\sum_{r\ge1}\pi_\alpha(r)^2\,r^{1+\sigma}}
{\displaystyle\frac{\Gamma(1-\alpha)^2}{\Gamma(1-2\alpha)}+\frac{\Gamma(1+2\alpha)}{\Gamma(1+\alpha)^2}-1}.
\end{equation}
Moreover, the numerator is at most $\alpha^2\zeta(1+2\alpha-\sigma)$, which gives the closed form
\begin{equation}\label{eq:Ccf}
C^U(\sigma)\;\le\;C^{\rm cf}(\sigma):=\min_{\sigma/2<\alpha<1/2}\;
\frac{\alpha^2\,\zeta(1+2\alpha-\sigma)}
{\displaystyle\frac{\Gamma(1-\alpha)^2}{\Gamma(1-2\alpha)}+\frac{\Gamma(1+2\alpha)}{\Gamma(1+\alpha)^2}-1},
\end{equation}
and, for $0<\sigma<\tfrac12$,
\begin{equation}\label{eq:simple}
C^{\rm cf}(\sigma)\;\le\;\sigma^2\,\zeta(1+\sigma)\,\frac{\Gamma(1-2\sigma)}{\Gamma(1-\sigma)^2}
\;=\;\sigma+\gamma_E\,\sigma^2+O(\sigma^3),
\end{equation}
where $\gamma_E=0.577\ldots$ is the Euler--Mascheroni constant. Consequently, as $\sigma\to0$,
\begin{equation}\label{eq:asym}
\begin{aligned}
1\;&\le\;\zeta(1+\sigma)\,\Cor(\sigma)\;\le\;1+2\gamma_E\,\sigma+O(\sigma^2),\\[2pt]
\tfrac12\;&\le\;\zeta(1+\sigma)\,C_c(\sigma)\;\le\;1+2\gamma_E\,\sigma+O(\sigma^2).
\end{aligned}
\end{equation}
\end{corollary}

Lemma~\ref{lem:sibuya} and corollary~\ref{cor:bound} are proved in section~\ref{sec:sibuya}. The bound $C^U$ is
shown in figure~\ref{fig:bounds}, together with the lower bounds of proposition~\ref{prop:lower}.

Since $\zeta(1+\sigma)=1/\sigma+\gamma_E+O(\sigma)$, the bounds \eqref{eq:asym} mean that the oriented threshold is
$\Cor=\sigma+O(\sigma^2)$, and that for the unoriented model the ratio $C_c/\sigma$ lies between $1/2$ and $1$ in
the limit $\sigma\to0$.

The result for the oriented model has a simple interpretation. A site has on average $C\zeta(1+\sigma)$ open bonds
to sites on its right. If the sites that are reached from $0$ along different paths were all different, as on a
tree, the forward cluster would be the family tree of a branching process in which every site has on average
$C\zeta(1+\sigma)$ offspring. Such a process survives forever with positive probability exactly when the mean
number of offspring exceeds one, that is, for $C>1/\zeta(1+\sigma)$. This is the mean-field value of the
threshold. Proposition~\ref{prop:lower} states that the true threshold is not lower, and \eqref{eq:asym} shows that
for small $\sigma$ it is only slightly higher. In this sense the oriented transition is mean-field-like for small
$\sigma$.

\subsection{Coarse-graining}

The bound $C^U(\sigma)$ is useful only when it is below $1$, which is the case for $\sigma<0.667$
(section~\ref{sec:numerics}). The reason is the cost of the shared bonds. In theorem~\ref{thm:main} a shared bond
of length $r$ costs $1/p_r-1=r^{1+\sigma}/C-1$. This is at least $1/C-1$, the cost of the shortest bond, and it
grows rapidly with $r$, the more so the larger $\sigma$ is.

To reduce the cost we replace the sites by blocks of $L$ consecutive sites (figure~\ref{fig:blocks}). Consider two
blocks that are $r$ blocks apart; in figure~\ref{fig:blocks}, for example, $B_0$ and $B_2$ are two blocks apart.
There are $L^2$ pairs of sites with one site in each block. The distance between the two sites of a pair is about
$rL$ sites, so each pair is connected with probability about
$C/(rL)^{1+\sigma}$. The mean number of bonds between the two blocks is therefore about
\[
L^2\cdot\frac{C}{(rL)^{1+\sigma}}=\frac{C\,L^{1-\sigma}}{r^{1+\sigma}},
\]
which is larger by the factor $L^{1-\sigma}$ than the mean number of bonds, $C/r^{1+\sigma}$, between two single
sites at distance $r$. A connection between two blocks is therefore much more likely than a connection between two
sites, and it costs less when two paths share it. This gain is not free: a connection between two blocks is of use
only if each of the blocks is connected inside, and this requirement adds a term to the bound.

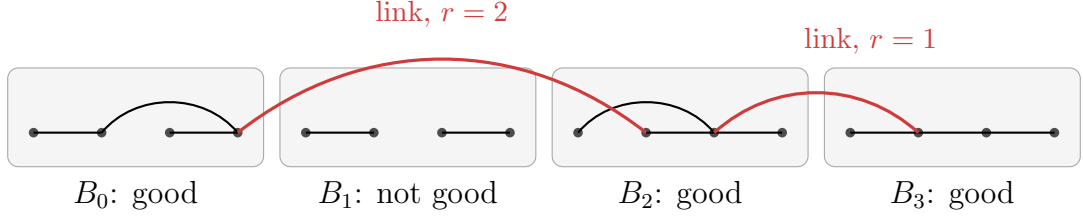
\begin{figure}[t]
\centering
\begin{tikzpicture}[x=0.9cm,y=0.9cm]
  \definecolor{acc}{RGB}{208,59,59}
  \foreach \b in {0,1,2,3} \draw[rounded corners=4pt,gray!70,fill=gray!8] (4*\b-0.38,-0.5) rectangle (4*\b+3.38,0.95);
  \foreach \i in {0,...,15} \fill[black!70] (\i,0) circle (1.8pt);
  \foreach \a/\b in {0/1,2/3,4/5,6/7,9/10,10/11,12/13,13/14,14/15} \draw[thick] (\a,0) -- (\b,0);
  \draw[thick] (1,0) to[bend left=50] (3,0);
  \draw[thick] (8,0) to[bend left=50] (10,0);
  \draw[very thick,acc] (3,0) to[bend left=38] (9,0);
  \draw[very thick,acc] (10,0) to[bend left=42] (13,0);
  \node[acc] at (6,1.75) {\small link, $r=2$};
  \node[acc] at (12.3,1.35) {\small link, $r=1$};
  \node[below] at (1.5,-0.55) {$B_0$: good};
  \node[below] at (5.5,-0.55) {$B_1$: not good};
  \node[below] at (9.5,-0.55) {$B_2$: good};
  \node[below] at (13.5,-0.55) {$B_3$: good};
\end{tikzpicture}
\caption{Blocks of $L=4$ sites. Black lines are open bonds inside the blocks. The blocks $B_0$, $B_2$ and $B_3$
are good: every two of their sites are connected by a path inside the block. The block $B_1$ is not good, because
inside it the sites split into two pairs. Red lines are open bonds between different blocks: $B_0$ is linked to
$B_2$, at distance $r=2$ in units of blocks, and $B_2$ is linked to $B_3$. The blocks $B_0$, $B_2$, $B_3$ form an
open path of blocks, and all their sites belong to one cluster.}
\label{fig:blocks}
\end{figure}

Let $B_i=\{iL,\,iL+1,\dots,(i+1)L-1\}$ be the $i$th block. A block is \emph{good} if every two of its sites are
connected by an open path that stays inside the block, that is, a path that uses only bonds between sites of the
block. It is not required that every two sites of the block are connected directly by a bond. We denote by $g_L$
the probability that a block is good. Two blocks $B_i$ and $B_j$ are \emph{linked} if at least one bond between a
site of $B_i$ and a site of $B_j$ is open. For $j-i=r$ this has the probability
\begin{equation}\label{eq:pL}
p^{(L)}_r=1-\prod_{x\in B_0,\;y\in B_r}\bigl(1-p_{y-x}\bigr).
\end{equation}

\begin{theorem}\label{thm:block}
For all $L\ge1$, all $\pi$ with $W(\pi)<\infty$ and all $0<C\le1$,
\begin{equation}\label{eq:block}
\theta(C)\ge 1-W(\pi)\Bigl[(1-g_L)+\sum_{r\ge1}\pi(r)^2\Bigl(\frac{1}{p^{(L)}_r}-1\Bigr)\Bigr].
\end{equation}
\end{theorem}

Here the random monotone path is a path of blocks. It is built as before, with steps drawn from $\pi$, but a step
of length $r$ now moves the path from a block $B_i$ to the block $B_{i+r}$: in theorem~\ref{thm:block} step lengths
are counted in blocks, as in $p^{(L)}_r$. The sum in \eqref{eq:block} is the mean cost of a shared link
between blocks, and the new term $1-g_L$ is the cost of the requirement that the blocks are good. For $L=1$ this
is theorem~\ref{thm:main}, since $g_1=1$ and $p^{(1)}_r=p_r$. For $L\ge2$ the blocks are connected
inside by bonds in both directions, so theorem~\ref{thm:block} concerns the unoriented model only.

\begin{corollary}\label{cor:less1}
Let $0<\sigma<1$ and $\sigma/2<\alpha<\frac12$, and let $L\ge2$ be the smallest integer with
$L^{1-\sigma}\ge4\,W(\alpha)\,\alpha^2\zeta(1+2\alpha-\sigma)$. Then
\[
C_c(\sigma)\le1-\frac{1}{2\,W(\alpha)\,(L-1)}<1 .
\]
\end{corollary}

Theorem~\ref{thm:block} and corollary~\ref{cor:less1} are proved in section~\ref{sec:blocks}. That $C_c<1$ for all
$\sigma<1$ was proved in~\cite{newman_one_1986}, but without a numerical value; corollary~\ref{cor:less1} gives one. This
value is very close to $1$, and a numerical evaluation of \eqref{eq:block} gives much better bounds
(section~\ref{sec:numerics}).

\subsection{Comparison with the numerical thresholds}

\begin{table}[t]
\centering
\caption{Bounds on the threshold $C_c$ of LR percolation. The upper bounds are rounded up and the lower bounds
down, so that the printed values are still bounds. $C_S$: the lower bound
\eqref{eq:schulman} of Schulman~\cite{schulman_long_1983}. $C_c$, numerical: the Monte Carlo values from table~I
of~\cite{gori_one-dimensional_2017}. Upper bound: the bound of theorem~\ref{thm:block} with blocks of $L$ sites and Sibuya steps,
minimized over $\alpha$ and over $L$. For $L=1$ (no blocks) it is $C^U$ of corollary~\ref{cor:bound}, and for
$\sigma=0.9$ it is the bound of corollary~\ref{cor:less1}.}
\label{tab:numbers}
\medskip
\begin{tabular}{c@{\qquad}c@{\qquad}c@{\qquad}c@{\qquad}c}
\toprule
$\sigma$ & $C_S$, lower bound & $C_c$, numerical & upper bound & $L$ \\
{\small } & {\small (Schulman~\cite{schulman_long_1983})} & {\small (Gori \textit{et al}~\cite{gori_one-dimensional_2017})}
 & {\small (cor.~\ref{cor:bound}, thm.~\ref{thm:block})} & {\small } \\
\midrule
0.1 & 0.0472 & 0.0477 & 0.0950 & 1 \\
0.2 & 0.0894 & 0.0932 & 0.1851 & 1 \\
0.3 & 0.1271 & 0.1405 & 0.2814 & 1 \\
0.4 & 0.1610 & 0.1935 & 0.3969 & 1 \\
0.5 & 0.1913 & 0.2548 & 0.5501 & 1 \\
0.6 & 0.2187 & 0.3271 & 0.7747 & 1 \\
0.7 & 0.2433 & 0.4138 & 0.9696 & 8 \\
0.8 & 0.2656 & 0.5210 & 0.99967 & 400 \\
0.9 & 0.2857 & 0.6641 & $1-7\times10^{-23}$ & $10^{21}$ \\
\bottomrule
\end{tabular}
\end{table}

Figure~\ref{fig:bounds} and table~\ref{tab:numbers} compare the bounds with the numerical values of $C_c$ from
Gori \textit{et al}~\cite{gori_one-dimensional_2017}. For the oriented model the two bounds are very close for small $\sigma$ (figure~\ref{fig:bounds}): their ratio is
$1.0006$ at $\sigma=0.05$, $1.005$ at $\sigma=0.1$ and $1.035$ at $\sigma=0.2$. For the unoriented model the upper
bound is about twice the numerical value for small $\sigma$, and it deteriorates as $\sigma$ grows. The closed
form $C^{\rm cf}$ of \eqref{eq:Ccf} is weaker than $C^U$: it equals $0.106$, $0.228$, $0.377$, $0.578$ and $0.871$
for $\sigma=0.1,\dots,0.5$, and it exceeds $1$ for $\sigma\ge0.54$. The origin of the factor of two between the
upper bound and the actual threshold of the unoriented model, at small $\sigma$, is discussed in
section~\ref{sec:discussion}.

\section{Proof of theorem~\ref{thm:main}}
\label{sec:proof}

There are two independent sources of randomness: the bonds, and the random path. Below, $\PP$ and $\EE$ denote probability and
expectation with respect to the random bonds, and $\EE_\gamma$ denotes the average over the random path with the
bonds held fixed. Throughout, $W<\infty$, as assumed in theorem~\ref{thm:main}.

\step{1}{a quantity that detects open paths}
Fix $n\ge1$. We denote by $\gamma_n=(x_0,x_1,\dots,x_n)$, with $x_0=0$, a monotone path of $n$ steps from $0$, given by its
sites. (Lower-case letters $x_k$ are the sites of a given path, and capital letters $X_k$ the random sites of the
random monotone path of section~\ref{sec:results}.) Such a path uses $n$ different bonds. Bonds are independent,
so $\gamma_n$ is open with probability
\[
P(\gamma_n)=\prod_{k=1}^{n}p_{x_k-x_{k-1}}>0 .
\]
Since its steps are independent, the random monotone path, stopped after $n$ steps, is the path $\gamma_n$ with
probability $\pi(\gamma_n)=\prod_{k=1}^n\pi(x_k-x_{k-1})$. We define
\begin{equation}\label{eq:Zn}
Z_n=\EE_\gamma\Bigl[\frac{\ind\{\gamma_n\text{ open}\}}{P(\gamma_n)}\Bigr]
=\sum_{\gamma_n}\pi(\gamma_n)\,\frac{\ind\{\gamma_n\text{ open}\}}{P(\gamma_n)},
\end{equation}
where in the first expression $\gamma_n$ is the random path $(X_0,\dots,X_n)$, and the sum runs over all monotone
paths $\gamma_n$ of $n$ steps from $0$. Here $\ind\{A\}$ denotes the indicator of an
event $A$: it equals $1$ if $A$ occurs and $0$ otherwise. Thus a path contributes $\pi(\gamma_n)/P(\gamma_n)$ to
$Z_n$ if it is open, and nothing if it is not. As stated above, the average $\EE_\gamma$ in \eqref{eq:Zn} is over the
random path only, with the bonds held fixed. Hence $Z_n$ is still a random variable: its value depends on which bonds are
open. It counts the open paths, each with the weight $\pi(\gamma_n)/P(\gamma_n)$, and it has two properties.

(a) \emph{Its mean over the bonds is $1$.} The mean of an indicator is the probability of its event, so the mean
of $\ind\{\gamma_n\text{ open}\}$ over the bonds is $P(\gamma_n)$. Averaging \eqref{eq:Zn} over the bonds, term by
term, therefore gives
\[
\EE Z_n=\sum_{\gamma_n}\pi(\gamma_n)\,\frac{P(\gamma_n)}{P(\gamma_n)}=\sum_{\gamma_n}\pi(\gamma_n)=1 .
\]

(b) \emph{If $Z_n>0$, the forward cluster of $0$ contains a site at distance at least $n$.} Indeed, if $Z_n>0$ then
at least one term of the sum is not zero, so some monotone path of $n$ steps from $0$ is open. Each step has length
at least $1$, so the path ends at a site $x_n\ge n$.

\step{2}{the second-moment inequality}
For every random variable $Z\ge0$ with $0<\EE Z^2<\infty$,
\begin{equation}\label{eq:PZ}
\PP(Z>0)\;\ge\;\frac{(\EE Z)^2}{\EE Z^2}.
\end{equation}
Indeed, since $Z=Z\,\ind\{Z>0\}$, the Cauchy--Schwarz inequality gives
\[
(\EE Z)^2=\bigl(\EE\bigl[Z\,\ind\{Z>0\}\bigr]\bigr)^2\le\EE\bigl[Z^2\bigr]\;\EE\bigl[\ind\{Z>0\}^2\bigr]=\EE Z^2\;\PP(Z>0).
\]
With property (a), inequality~\eqref{eq:PZ} for $Z=Z_n$ reads
\begin{equation}\label{eq:PZn}
\PP(Z_n>0)\ge\frac{1}{\EE Z_n^2}.
\end{equation}
Hence, if $\EE Z_n^2\le V$ for every $n$, then $\PP(Z_n>0)\ge1/V$ for every $n$. Steps 3 to 8 find such a
constant $V$.

\step{3}{the second moment counts shared bonds}
We write $Z_n^2$ as a double sum over two monotone paths $\gamma_n$ and $\gamma'_n$ of $n$ steps from $0$,
\[
Z_n^2=\sum_{\gamma_n}\sum_{\gamma'_n}\pi(\gamma_n)\,\pi(\gamma'_n)\,
\frac{\ind\{\gamma_n\text{ open}\}\,\ind\{\gamma'_n\text{ open}\}}{P(\gamma_n)\,P(\gamma'_n)} .
\]
Both paths are open exactly when every bond of the union $\gamma_n\cup\gamma'_n$ is open. A bond that both paths use,
a \emph{shared} bond, appears only once in the union. Writing $p_e=p_{b-a}$ for the probability that a bond
$e=\{a,b\}$, $a<b$, is open,
\[
\PP(\gamma_n\text{ and }\gamma'_n\text{ are open})=\prod_{e\in\gamma_n\cup\gamma'_n}p_e
=\frac{P(\gamma_n)\,P(\gamma'_n)}{\prod_{e\in\gamma_n\cap\gamma'_n}p_e}.
\]
The product of the two indicators in the double sum is the indicator that both paths are open, and its mean is
this probability. Averaging the double sum over the bonds, the factors $P(\gamma_n)P(\gamma'_n)$ therefore cancel
and we obtain
\[
\EE Z_n^2=\sum_{\gamma_n}\sum_{\gamma'_n}\pi(\gamma_n)\,\pi(\gamma'_n)\prod_{e\in\gamma_n\cap\gamma'_n}\frac1{p_e}
=\EE_{\gamma,\gamma'}\Bigl[\prod_{e\in\gamma_n\cap\gamma'_n}\frac1{p_e}\Bigr].
\]
Here $\gamma$ and $\gamma'$ are two \emph{independent} random monotone paths, and $\EE_{\gamma,\gamma'}$ is the
average over both. In words, every shared bond of length $r$ contributes a factor $1/p_r\ge1$, and a pair of
paths without shared bonds contributes the factor $1$ (a product over an empty set). If no two paths shared a bond,
we would have $\EE Z_n^2=1$ and $\PP(Z_n>0)=1$.

The bonds shared by the first $n$ steps are among the bonds shared by the full infinite paths, and every factor is
at least $1$. Therefore
\begin{equation}\label{eq:V}
\EE Z_n^2\;\le\;V:=\EE_{\gamma,\gamma'}\Bigl[\prod_{e\in\gamma\cap\gamma'}\frac1{p_e}\Bigr]\qquad\text{for every }n,
\end{equation}
and by \eqref{eq:PZn}, $\PP(Z_n>0)\ge1/V$ for every $n$.

\step{4}{from finite $n$ to an infinite cluster}
If a path of $n+1$ steps is open, its first $n$ steps form an open path of $n$ steps. Hence $Z_{n+1}>0$ implies
$Z_n>0$, that is, the events $\{Z_n>0\}$ \emph{decrease} with $n$. For \emph{decreasing} events the probability of the
intersection is the limit of the probabilities, so
\[
\PP(Z_n>0\text{ for all }n)=\lim_{n\to\infty}\PP(Z_n>0)\ge\frac1V .
\]
If $Z_n>0$ for all $n$, then by property (b) the forward cluster of $0$ contains sites at arbitrarily large
distance, so it is infinite. Therefore
\begin{equation}\label{eq:theta}
\thor(C)\ge\frac1V .
\end{equation}
The rest of the proof is the computation of $V$.

\begin{figure}[t]
\centering
\begin{tikzpicture}[x=0.76cm,y=0.76cm]
  \definecolor{acc}{RGB}{208,59,59}
  \foreach \c in {0,5,9} \draw[densely dotted,gray] (\c,-1.9) -- (\c,1.75);
  \draw[gray!55] (-0.5,0) -- (19.7,0);
  \foreach \i in {0,...,19} \fill[black!60] (\i,0) circle (1.6pt);
  \foreach \a/\b in {0/3,3/5,9/12,12/17} \draw[thick] (\a,0) to[bend left=55] (\b,0);
  \draw[thick] (17,0) to[bend left=28] (19.7,0.95);
  \foreach \a/\b in {0/2,2/5,9/10,10/14,14/18} \draw[thick,dashed] (\a,0) to[bend right=55] (\b,0);
  \draw[thick,dashed] (18,0) to[bend right=25] (19.7,-0.55);
  \draw[very thick,acc] (5,0) to[bend left=55] (9,0);
  \draw[very thick,acc,dashed] (5,0) to[bend right=55] (9,0);
  \foreach \c in {0,5,9} \fill[acc] (\c,0) circle (3pt);
  \node[above] at (0,1.75) {$c_0=0$};
  \node[above] at (5,1.75) {$c_1$};
  \node[above] at (9,1.75) {$c_2$};
  \node at (1.5,1.05) {$\gamma$};
  \node at (3.5,-1.15) {$\gamma'$};
  \node[acc] at (7,1.4) {\small shared bond};
  \draw[decorate,decoration={brace,mirror,amplitude=4pt}] (0.15,-2.0) -- (4.85,-2.0)
     node[midway,below=5pt] {\small (B): weight $1$};
  \draw[decorate,decoration={brace,mirror,amplitude=4pt}] (5.15,-2.0) -- (8.85,-2.0)
     node[midway,below=5pt] {\small (A): weight $1/p_4$};
  \draw[decorate,decoration={brace,mirror,amplitude=4pt}] (9.15,-2.0) -- (19.7,-2.0)
     node[midway,below=5pt] {\small (C): the paths do not meet again};
\end{tikzpicture}
\caption{Two independent monotone paths, $\gamma$ (solid, drawn above the lattice) and $\gamma'$ (dashed, drawn
below). The common sites are $c_0=0$, $c_1=5$ and $c_2=9$, so $N=3$ (step~5). From $c_0$ to $c_1$ the paths use different bonds. From
$c_1$ to $c_2$ both paths make the same jump of length $4$, so they share this bond. After $c_2$ they never meet
again. The cases (A), (B) and (C) are defined in step~6. In this example the product in \eqref{eq:V} equals
$1/p_4$.}
\label{fig:paths}
\end{figure}
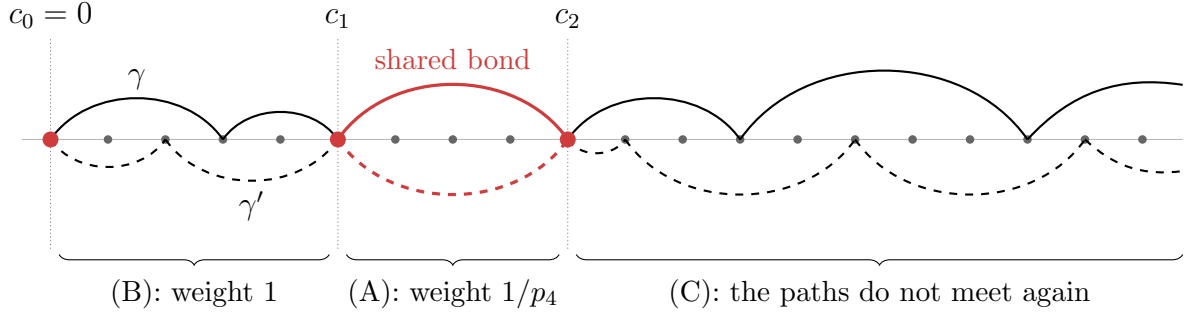

\step{5}{where two paths can share a bond}
We call a site \emph{common} if both paths visit it. Site $0$ is common. Let
\[
c_0<c_1<c_2<\cdots
\]
be the common sites in increasing order, so that $c_0=0$, and let $N$ be their number (figure~\ref{fig:paths}).

A path uses the bond $\{a,b\}$, $a<b$, exactly when it visits $a$ and then jumps from $a$ to $b$ in one step. In
that case it visits no site strictly between $a$ and $b$. So if both paths use $\{a,b\}$, then $a$ and $b$ are
common sites and there is no common site between them: $a=c_i$ and $b=c_{i+1}$ for some $i$. Conversely, if both
paths jump from $c_i$ to $c_{i+1}$ in one step, they share that bond. Thus the shared bonds are exactly the pairs of
consecutive common sites $c_i,c_{i+1}$ that \emph{both} paths join by a single jump.

We call the stretch from $c_i$ to $c_{i+1}$ a \emph{segment}, and give it the weight
\[
w_i=\begin{cases}1/p_r & \text{if both paths jump from $c_i$ to $c_{i+1}$ in one step, of length } r=c_{i+1}-c_i,\\[2pt]
1 & \text{otherwise.}\end{cases}
\]
Then the product in \eqref{eq:V} is the product of the weights of all segments,
\begin{equation}\label{eq:segments}
\prod_{e\in\gamma\cap\gamma'}\frac1{p_e}=\prod_{i=0}^{N-2}w_i .
\end{equation}

\step{6}{the paths start afresh at every common site}
Since the steps of a path are independent, the part of a path after a site $x$ that it visits is a new random
monotone path from $x$, independent of the part before $x$. This is the \emph{renewal property}. Hence, after a
common site $x$, the two paths continue as two independent random monotone paths from $x$, independent of
everything to the left of $x$.

Therefore, after every common site the same thing happens as after site $0$. Exactly one of three cases occurs:
\begin{itemize}
\item[(A)] Both paths make the same jump, of some length $r$. The probability is $\pi(r)^2$. The next common site
is then at distance $r$, and the segment is a shared bond with weight $1/p_r$.
\item[(B)] The paths meet again, but not as in (A). The segment has weight $1$.
\item[(C)] The paths never meet again. There is no further segment.
\end{itemize}
Let $\rho$ be the probability that two independent paths that start at the same site meet again, that is, the
probability of (A) or (B). The three cases then have the probabilities
\begin{equation}\label{eq:cases}
\text{(A): }\sum_{r\ge1}\pi(r)^2,\qquad\text{(B): }\rho-\sum_{r\ge1}\pi(r)^2,\qquad\text{(C): }1-\rho .
\end{equation}

\step{7}{the return probability is $\rho=1-1/W$}
By step~6, after every common site there is a further one with probability $\rho$, independently of what happened
to its left. So
the number $N$ of common sites satisfies $\PP(N\ge k+1)=\rho^{k}$ for $k=0,1,2,\dots$, and
\[
\EE N=\sum_{k\ge0}\PP(N\ge k+1)=\sum_{k\ge0}\rho^k=\frac1{1-\rho}.
\]
On the other hand, $N$ is the number of sites that both paths visit. The paths are independent, so the probability
that both visit $x$ is $u(x)^2$, and $\EE N=\sum_{x\ge0}u(x)^2=W$ by \eqref{eq:W}. Comparing the two expressions,
\begin{equation}\label{eq:rho}
\rho=1-\frac1W .
\end{equation}
Since $W<\infty$ we have $\rho<1$, and the number of common sites is finite.

\step{8}{summing over the segments}
By \eqref{eq:segments}, $V$ is the average of the product of the weights of the segments. We split this average
according to the number $k=N-1$ of segments:
\[
V=\sum_{k\ge0}\EE_{\gamma,\gamma'}\Bigl[\,\ind\{N=k+1\}\prod_{i=0}^{k-1}w_i\Bigr].
\]
A pair of paths has exactly $k$ segments when case (A) or (B) occurs at each of the first $k$ common sites, and case
(C) at the next one. By step~6 these cases are independent from one common site to the next. Each of the $k$
segments therefore contributes the factor
\[
m=\underbrace{\sum_{r\ge1}\pi(r)^2\,\frac1{p_r}}_{\text{(A)}}+\underbrace{\Bigl(\rho-\sum_{r\ge1}\pi(r)^2\Bigr)\cdot1}_{\text{(B)}}
=\rho+S,\qquad S:=\sum_{r\ge1}\pi(r)^2\Bigl(\frac1{p_r}-1\Bigr),
\]
which is its weight averaged over the cases (A) and (B), with their probabilities \eqref{eq:cases}. The final case
(C) adds no segment and contributes its probability $1-\rho$. Hence the term with $k$ segments is $m^k(1-\rho)$, and
\[
V=\sum_{k\ge0}m^k\,(1-\rho)=\frac{1-\rho}{1-m}\qquad\text{if }m<1 .
\]
By \eqref{eq:rho}, $1-\rho=1/W$ and $1-m=1-\rho-S=1/W-S$. Thus
\begin{equation}\label{eq:Vfinal}
V=\frac{1/W}{1/W-S}=\frac{1}{1-WS}\qquad\text{if }WS<1 .
\end{equation}

\step{9}{conclusion}
If $WS<1$, then by \eqref{eq:theta} and \eqref{eq:Vfinal}
\[
\thor(C)\ge\frac1V=1-WS=1-W(\pi)\sum_{r\ge1}\pi(r)^2\Bigl(\frac1{p_r}-1\Bigr).
\]
If $WS\ge1$ the right-hand side is not positive, and the bound holds trivially because $\thor(C)\ge0$. Finally $\theta\ge\thor$, which
completes the proof of theorem~\ref{thm:main}. \hfill$\square$

\section{Sibuya steps: proofs of lemma~\ref{lem:sibuya} and corollary~\ref{cor:bound}}
\label{sec:sibuya}

\begin{proof}[Proof of lemma~\ref{lem:sibuya}]
\textbf{(S1)} The binomial series is $(1-z)^\alpha=\sum_{r\ge0}\binom{\alpha}{r}(-z)^r$. Comparing the coefficients
of $z^r$ in \eqref{eq:gf},
\begin{align*}
\pi_\alpha(r)=(-1)^{r+1}\binom{\alpha}{r}&=(-1)^{r+1}\,\frac{\alpha(\alpha-1)(\alpha-2)\cdots(\alpha-r+1)}{r!}\\[3pt]
&=\frac{\alpha(1-\alpha)(2-\alpha)\cdots(r-1-\alpha)}{r!}
=\frac{\alpha}{r}\prod_{k=1}^{r-1}\Bigl(1-\frac{\alpha}{k}\Bigr).
\end{align*}
In the second line each of the $r-1$ negative factors $(\alpha-1),\dots,(\alpha-r+1)$ was multiplied by $-1$,
which cancels the sign $(-1)^{r+1}$. All $\pi_\alpha(r)$ are positive, and setting $z=1$ in \eqref{eq:gf} shows that
they sum to $1$.

\textbf{(S2)} We use $1-x\le e^{-x}$ in every factor of the product, and then $\sum_{k=1}^{r-1}1/k\ge\ln r$, which
follows by comparing the sum with the integral of $1/x$ from $1$ to $r$:
\[
\prod_{k=1}^{r-1}\Bigl(1-\frac{\alpha}{k}\Bigr)\le\exp\Bigl(-\alpha\sum_{k=1}^{r-1}\frac1k\Bigr)\le e^{-\alpha\ln r}=r^{-\alpha},
\qquad\text{hence}\qquad \pi_\alpha(r)\le\alpha\,r^{-1-\alpha}.
\]

\textbf{(S3)} To visit a site $x\ge1$ the path must arrive there by a last step, of some length $r$, from the site
$x-r$. Hence, for every step distribution $\pi$,
\begin{equation}\label{eq:renewal}
u(x)=\sum_{r=1}^{x}\pi(r)\,u(x-r)\quad(x\ge1),\qquad u(0)=1 .
\end{equation}
We use the generating functions
\[
U(z)=\sum_{x\ge0}u(x)z^x=1+\sum_{x\ge1}u(x)z^x,\qquad \hat\pi(z)=\sum_{r\ge1}\pi(r)z^r ,
\]
where the term $x=0$ of $U$ is $u(0)=1$. Multiplying \eqref{eq:renewal} by $z^x=z^r\,z^{x-r}$ and summing over
$x\ge1$, the left-hand side gives $U(z)-1$, and on the right-hand side every pair $r\ge1$, $y=x-r\ge0$ appears
exactly once:
\begin{equation}\label{eq:U}
U(z)-1=\sum_{r\ge1}\pi(r)z^r\,\sum_{y\ge0}u(y)z^y=\hat\pi(z)\,U(z),\qquad\text{so that}\qquad U(z)=\frac{1}{1-\hat\pi(z)} .
\end{equation}
For Sibuya steps $1-\hat\pi(z)=(1-z)^\alpha$ by \eqref{eq:gf}, so \eqref{eq:U} gives $U(z)=(1-z)^{-\alpha}$. The
binomial series of $(1-z)^{-\alpha}$ then gives
\[
u(x)=\frac{\alpha(\alpha+1)\cdots(\alpha+x-1)}{x!}=\frac{\Gamma(x+\alpha)}{\Gamma(\alpha)\,x!}\;\approx\;\frac{x^{\alpha-1}}{\Gamma(\alpha)}
\quad\text{for large }x .
\]

\textbf{(S4)} Let $(a)_k=a(a+1)\cdots(a+k-1)$, with $(a)_0=1$. Gauss's summation theorem~\cite{andrews1999special} states that
\begin{equation}\label{eq:gauss}
\sum_{k\ge0}\frac{(a)_k\,(b)_k}{(c)_k\,k!}=\frac{\Gamma(c)\,\Gamma(c-a-b)}{\Gamma(c-a)\,\Gamma(c-b)}\qquad\text{for }c>a+b .
\end{equation}
By (S3), $u(x)=(\alpha)_x/x!$. Since $(1)_x=x!$, the sum $W$ of \eqref{eq:W} is the case $a=b=\alpha$, $c=1$ of \eqref{eq:gauss}:
\[
W(\alpha)=\sum_{x\ge0}\Bigl(\frac{(\alpha)_x}{x!}\Bigr)^2=\frac{\Gamma(1-2\alpha)}{\Gamma(1-\alpha)^2}\qquad\text{for }0<\alpha<\tfrac12 .
\]
For $\alpha\ge\frac12$ the sum diverges, because $u(x)^2\approx x^{2\alpha-2}/\Gamma(\alpha)^2$. In the same way,
by (S1), $\pi_\alpha(r)=-(-\alpha)_r/r!$ for $r\ge1$, while $(-\alpha)_0/0!=1$. The case $a=b=-\alpha$, $c=1$ of
\eqref{eq:gauss} gives
\[
1+\sum_{r\ge1}\pi_\alpha(r)^2=\sum_{r\ge0}\Bigl(\frac{(-\alpha)_r}{r!}\Bigr)^2=\frac{\Gamma(1+2\alpha)}{\Gamma(1+\alpha)^2},
\]
which is the second formula of \eqref{eq:Wsib}.
\end{proof}

\begin{proof}[Proof of corollary~\ref{cor:bound}]
We fix $\alpha$ in the window \eqref{eq:window} and apply theorem~\ref{thm:main} to $\pi=\pi_\alpha$. Since
$1/p_r=r^{1+\sigma}/C$,
\[
S=\sum_{r\ge1}\pi_\alpha(r)^2\Bigl(\frac1{p_r}-1\Bigr)=\frac1C\sum_{r\ge1}\pi_\alpha(r)^2\,r^{1+\sigma}-\sum_{r\ge1}\pi_\alpha(r)^2 .
\]
The right-hand side of \eqref{eq:main} is positive when $S<1/W$, that is, when
\[
\frac1C\sum_{r\ge1}\pi_\alpha(r)^2\,r^{1+\sigma}<\frac1W+\sum_{r\ge1}\pi_\alpha(r)^2
\qquad\Longleftrightarrow\qquad
C>\frac{\sum_{r\ge1}\pi_\alpha(r)^2\,r^{1+\sigma}}{\dfrac1W+\sum_{r\ge1}\pi_\alpha(r)^2}.
\]
Hence $\thor(C)>0$ for every $C$ above this ratio, and $\Cor$ is at most the ratio. Inserting \eqref{eq:Wsib}
gives the expression in \eqref{eq:CU} for each $\alpha$, and we may take the best $\alpha$. The condition
$\alpha>\sigma/2$ makes the numerator finite: by (S2), $\pi_\alpha(r)^2\,r^{1+\sigma}\le\alpha^2r^{-1-(2\alpha-\sigma)}$,
which is summable when $2\alpha-\sigma>0$ (this is the computation below). Conversely, (S1) gives
$\pi_\alpha(r)\approx\alpha\,r^{-1-\alpha}/\Gamma(1-\alpha)$ for large $r$, so the numerator is infinite when
$\alpha\le\sigma/2$. In the same way, by (S4), $W$ is finite if and only if $\alpha<\frac12$.

For the closed form \eqref{eq:Ccf} we use (S2) in the numerator,
\[
\sum_{r\ge1}\pi_\alpha(r)^2\,r^{1+\sigma}\le\alpha^2\sum_{r\ge1}r^{-2-2\alpha}\,r^{1+\sigma}
=\alpha^2\sum_{r\ge1}r^{-(1+2\alpha-\sigma)}=\alpha^2\,\zeta(1+2\alpha-\sigma).
\]
For \eqref{eq:simple} we choose $\alpha=\sigma$, which lies in the window when $\sigma<\frac12$, and drop the
nonnegative term $\Gamma(1+2\alpha)/\Gamma(1+\alpha)^2-1$ from the denominator, which only increases the ratio.
What remains is $\sigma^2\zeta(1+\sigma)\,W(\sigma)$. The expansion follows from
$\zeta(1+\sigma)=1/\sigma+\gamma_E+O(\sigma)$ and $\Gamma(1-2\sigma)/\Gamma(1-\sigma)^2=1+O(\sigma^2)$.

Finally, the lower bounds in \eqref{eq:asym} are those of proposition~\ref{prop:lower}, and the upper bounds follow
from \eqref{eq:order} and \eqref{eq:simple}:
$\zeta(1+\sigma)\,\Cor\le[\sigma\zeta(1+\sigma)]^2\,W(\sigma)=1+2\gamma_E\,\sigma+O(\sigma^2)$.
\end{proof}

\section{Blocks: proofs of theorem~\ref{thm:block} and corollary~\ref{cor:less1}}
\label{sec:blocks}

\begin{proof}[Proof of theorem~\ref{thm:block}]
The proof follows the nine steps of section~\ref{sec:proof}, and we list only the changes.

\emph{Paths.} A random monotone path $X_0<X_1<\cdots$, with $X_0=0$, is now a sequence of block indices. We call its first $n$
steps $\gamma_n$ open if (i) each of the $n+1$ blocks $B_{X_0},\dots,B_{X_n}$ is good, and (ii) each of these blocks
is linked to the next one. These events concern disjoint sets of bonds: the bonds inside different blocks, and the
bonds between different pairs of blocks. They are therefore independent, and
\[
P(\gamma_n)=g_L^{\,n+1}\prod_{k=1}^{n}p^{(L)}_{X_k-X_{k-1}} .
\]

\emph{Step 1.} $Z_n$ is defined as in \eqref{eq:Zn}, and $\EE Z_n=1$. If $Z_n>0$, some block path of $n$ steps is
open. All the sites of its blocks then belong to one cluster. This cluster contains site $0$, which lies in $B_0$,
and it contains sites at distance at least $nL$.

\emph{Step 2} is unchanged.

\emph{Step 3.} Both paths are open when all blocks of the union are good and all links of the union are present. A
common block and a shared link are counted once. Dividing by $P(\gamma_n)P(\gamma'_n)$ leaves a factor $1/g_L$ for
every common block and a factor $1/p^{(L)}_r$ for every shared link of length $r$. All factors are at least $1$, so
\[
\EE Z_n^2\le V_L:=\EE_{\gamma,\gamma'}\Bigl[\Bigl(\frac{1}{g_L}\Bigr)^{N}\prod_{\text{shared links}}\frac{1}{p^{(L)}_r}\Bigr],
\]
where $N$ is the number of common blocks of the two infinite paths, block $0$ included.

\emph{Step 4} gives $\theta(C)\ge1/V_L$.

\emph{Steps 5 to 7} concern only the two random paths, so they hold with `block' in place of `site'. In particular
$\rho=1-1/W$.

\emph{Step 8.} The mean weight of a segment is now $m_L=\rho+S_L$, with
$S_L=\sum_{r}\pi(r)^2\bigl(1/p^{(L)}_r-1\bigr)$. Pairs of paths with $k$ segments have $k+1$ common blocks, hence
\[
V_L=\sum_{k\ge0}\Bigl(\frac1{g_L}\Bigr)^{k+1}m_L^{\,k}\,(1-\rho)=\frac{1-\rho}{g_L-m_L}\qquad\text{if }m_L<g_L .
\]
With $1-\rho=1/W$ and $g_L-m_L=g_L-1+1/W-S_L$,
\[
\frac1{V_L}=W\Bigl(\frac1W-(1-g_L)-S_L\Bigr)=1-W\bigl[(1-g_L)+S_L\bigr],
\]
which is the right-hand side of \eqref{eq:block}. If $m_L\ge g_L$ this expression is not positive, and
\eqref{eq:block} holds trivially.
\end{proof}

To use theorem~\ref{thm:block} we need a lower bound on $g_L$ and an upper bound on $1/p^{(L)}_r-1$.

\begin{lemma}\label{lem:ingredients}
Let $D_r=\sum_{x\in B_0,\,y\in B_r}(y-x)^{-1-\sigma}$, so that $C\,D_r$ is the mean number of open bonds between two
blocks at distance $r$. Then:
\begin{itemize}
\item[(B1)] $g_L\ge C^{L-1}$, and also
$g_L\ge\prod_{t=1}^{L-1}\bigl[1-\prod_{j=1}^{t}(1-p_j)\bigr]$.
\item[(B2)] $D_r\ge L^{1-\sigma}\,r^{-1-\sigma}$.
\item[(B3)] $\displaystyle\frac{1}{p^{(L)}_r}-1\le\frac{1}{e^{CD_r}-1}\le\frac{1}{C\,D_r}\le\frac{r^{1+\sigma}}{C\,L^{1-\sigma}}$.
\end{itemize}
\end{lemma}

\begin{proof}
\textbf{(B1)} If the $L-1$ nearest-neighbour bonds of a block are open, the block is good, and each of them is open
with probability $p_1=C$, independently. Hence $g_L\ge C^{L-1}$. For the second bound, the block is good if every site of the block except the first has
an open bond to at least one of the sites of the block on its left. For the site with $t$ sites on its left this
has probability $1-\prod_{j=1}^{t}(1-p_j)$, and different sites involve different bonds.

\textbf{(B2)} The function $t\mapsto t^{-1-\sigma}$ is convex, so by Jensen's inequality its average over the $L^2$
pairs $(x,y)$ is at least its value at the average distance. The average of $y-x$ is exactly $rL$. Hence
$D_r\ge L^2\,(rL)^{-1-\sigma}=L^{1-\sigma}r^{-1-\sigma}$.

\textbf{(B3)} Using $1-x\le e^{-x}$ in every factor of \eqref{eq:pL}, $1-p^{(L)}_r\le e^{-C D_r}$. Therefore
\[
\frac{1}{p^{(L)}_r}-1=\frac{1-p^{(L)}_r}{p^{(L)}_r}\le\frac{e^{-CD_r}}{1-e^{-CD_r}}=\frac{1}{e^{CD_r}-1}\le\frac{1}{C\,D_r}.
\]
The first inequality holds because $q/(1-q)$ increases with $q$, and the second because $e^{a}-1\ge a$. The last
inequality of (B3) is (B2).
\end{proof}

Comparing (B3) with the cost $1/p_r-1\le r^{1+\sigma}/C$ of a single bond, we see that a link between blocks costs
$L^{1-\sigma}$ times less.

\begin{proof}[Proof of corollary~\ref{cor:less1}]
We take Sibuya steps with $\alpha$ in the window \eqref{eq:window}, and write
$A=\alpha^2\zeta(1+2\alpha-\sigma)$ and $W=W(\alpha)$. We bound the two terms in the bracket of \eqref{eq:block}. By
(B1) and Bernoulli's inequality, $C^{L-1}\ge1-(L-1)(1-C)$,
\[
1-g_L\le1-C^{L-1}\le(L-1)(1-C).
\]
By (B3), and then by (S2) as in the proof of corollary~\ref{cor:bound},
\[
\sum_{r\ge1}\pi_\alpha(r)^2\Bigl(\frac{1}{p^{(L)}_r}-1\Bigr)\le\frac{1}{C\,L^{1-\sigma}}\sum_{r\ge1}\pi_\alpha(r)^2\,r^{1+\sigma}
\le\frac{A}{C\,L^{1-\sigma}} .
\]
Let $L\ge2$ be the smallest integer with $L^{1-\sigma}\ge4WA$. For $C\ge\frac12$ the second term is then at most
$2A/L^{1-\sigma}\le1/(2W)$. The first term is less than $1/(2W)$ when $1-C<1/[2W(L-1)]$. For such $C$, which
are indeed at least $\frac12$ because $W\ge1$ and $L\ge2$, the bracket is less than $1/W$, and
theorem~\ref{thm:block} gives $\theta(C)>0$.
\end{proof}

With the best $\alpha$, corollary~\ref{cor:less1} gives $1-C_c\ge4\times10^{-5}$ at $\sigma=0.7$,
$5\times10^{-9}$ at $\sigma=0.8$ and $7\times10^{-23}$ at $\sigma=0.9$.

\section{Numerical evaluation}
\label{sec:numerics}

The proofs and the codes of this paper were developed with the help of a generative AI tool (see the AI disclosure
statement).

\subsection{Evaluation of the bounds}

The upper bounds in figure~\ref{fig:bounds} and table~\ref{tab:numbers} are numerical evaluations of
\eqref{eq:CU} and \eqref{eq:block}. No simulation is involved in them.

The sums over $r$ are evaluated up to $r=R=2\times10^5$. The rest is bounded from above by using
$\pi_\alpha(r)\le\pi_\alpha(R)\,(R/r)^{1+\alpha}$ for $r\ge R$. This follows from (S1) as (S2) does:
\[
\frac{\pi_\alpha(r)}{\pi_\alpha(R)}=\frac Rr\prod_{k=R}^{r-1}\Bigl(1-\frac{\alpha}{k}\Bigr)
\le\frac Rr\exp\Bigl(-\alpha\sum_{k=R}^{r-1}\frac1k\Bigr)\le\Bigl(\frac Rr\Bigr)^{1+\alpha},
\]
since $\sum_{k=R}^{r-1}1/k\ge\ln(r/R)$. Hence
\[
\sum_{r>R}\pi_\alpha(r)^2\,r^{1+\sigma}\le\frac{\pi_\alpha(R)^2\,R^{2+\sigma}}{2\alpha-\sigma}.
\]
Every computed value is therefore a rigorous upper bound on the threshold. The minimum over $\alpha$ in
\eqref{eq:CU} is found numerically; since every $\alpha$ in the window gives a valid bound, an inaccurate minimum
could only make the bound slightly weaker, never wrong. In table~\ref{tab:numbers} the upper bounds are rounded up
and the lower bounds down. We find $C^U(\sigma)<1$ for $\sigma<0.6672$. This end point is close to $2/3$ but not
equal to it: $C^U(2/3)=0.998$.

For theorem~\ref{thm:block} we use Sibuya steps, the second bound of (B1), and the bound $1/(e^{CD_r}-1)$ of
(B3), with $D_r$ computed from its definition (a finite double sum) for $r\le200$ and replaced by its lower bound
(B2) for larger $r$. For each $\sigma$ we find the smallest $C$ at which the right-hand side of \eqref{eq:block} is
positive, minimized over $\alpha$ and over the block sizes $2\le L\le5000$. Blocks improve on $C^U$ for
$\sigma>0.645$, and the best block size grows from $L=4$ at $\sigma=0.65$ to $L=2000$ at $\sigma=0.82$. For
$\sigma\ge0.84$ larger blocks would be needed; for $\sigma=0.9$, table~\ref{tab:numbers} gives the bound of
corollary~\ref{cor:less1}, with $\alpha=0.473$ and $L\approx10^{21}$.

\subsection{Numerical checks}

As an independent check, the formulas of this paper were tested numerically in three ways.
\begin{itemize}
\item[(i)] The closed forms \eqref{eq:Wsib} agree with a direct summation of the series to six digits, for
$\alpha=0.2$, $0.3$ and $0.4$.
\item[(ii)] The renewal formula \eqref{eq:Vfinal} holds for any weights in place of $1/p_r$. We tested it by a
Monte Carlo sampling of pairs of independent Sibuya paths with $\alpha=0.35$, with the weight $1.5$ for a shared
bond of length $r\le10$ and the weight $1$ otherwise. The formula gives $V=1.1260$, and $3\times10^5$ pairs of
paths, followed up to site $2\times10^6$, gave $1.1247\pm0.0006$. The cut-off of the paths lowers the Monte Carlo
value by about $5\times10^{-4}$.
\item[(iii)] In simulations of the model on a ring of $4\times10^5$ sites, the fraction of sites in the largest
cluster is well above the lower bound \eqref{eq:main} on $\theta$, as it should be. For example, at $\sigma=0.5$ and
$C=0.6$ the simulated fraction is $0.97$ and the bound is $\theta\ge0.10$, and at $\sigma=0.3$ and $C=0.3$ the
fraction is $0.87$ and the bound is $\theta\ge0.07$. Both points lie above the numerical thresholds
of~\cite{gori_one-dimensional_2017}, $0.2548$ and $0.1405$.
\end{itemize}

\section{Discussion}
\label{sec:discussion}

We have proved explicit upper bounds on the threshold of one-dimensional LR percolation. The first, $C^U(\sigma)$,
is given by a one-parameter minimization of a ratio of a convergent series and Gamma functions, and has the closed
forms \eqref{eq:Ccf} and \eqref{eq:simple}. The second is obtained from blocks of sites and is below $1$ for every
$\sigma<1$.

\textbf{The oriented model.} The paths that we use are monotone, and therefore theorem~\ref{thm:main} is in fact a
statement on oriented percolation. For this model the picture is rather complete for small $\sigma$: the threshold
is $\Cor=\sigma+O(\sigma^2)$, and the two bounds are much closer than \eqref{eq:asym} suggests. Numerically, the
ratio of the bounds, $\zeta(1+\sigma)\,C^U(\sigma)$, exceeds $1$ by about $5\,\sigma^3$ for $\sigma\le0.1$. Most of
this excess comes from the denominator of \eqref{eq:CU}. By \eqref{eq:cases} and \eqref{eq:rho}, one minus this
denominator is the probability of case (B) of step~6, in which two paths meet again without sharing a bond, and
expanding the Gamma functions gives
\[
1-\Bigl[\frac{\Gamma(1-\alpha)^2}{\Gamma(1-2\alpha)}+\frac{\Gamma(1+2\alpha)}{\Gamma(1+\alpha)^2}-1\Bigr]
=4\zeta(3)\,\alpha^3+O(\alpha^4),\qquad 4\zeta(3)=4.808\ldots
\]
The paths are so sparse that such meetings are a third-order effect. We have not tried to prove the corresponding
expansion of $C^U$.

It is also instructive to see directly why theorem~\ref{thm:main} cannot give more than
proposition~\ref{prop:lower} allows. By the Cauchy--Schwarz inequality,
\[
1=\Bigl(\sum_{r\ge1}\frac{\pi(r)}{\sqrt{p_r}}\,\sqrt{p_r}\Bigr)^2
\le\sum_{r\ge1}\frac{\pi(r)^2}{p_r}\;\sum_{r\ge1}p_r ,\qquad\text{so}\qquad
\sum_{r\ge1}\frac{\pi(r)^2}{p_r}\ge\frac{1}{C\,\zeta(1+\sigma)}.
\]
On the other hand, the right-hand side of \eqref{eq:main} is positive only if
$\sum_r\pi(r)^2/p_r<1/W+\sum_r\pi(r)^2$, and this is at most $1$, because $\sum_r\pi(r)^2\le\rho=1-1/W$ by
\eqref{eq:cases}. Hence theorem~\ref{thm:main} can prove percolation only for $C>1/\zeta(1+\sigma)$, whatever the
step distribution is. Equality in the Cauchy--Schwarz inequality requires $\pi(r)\propto p_r$, and the Sibuya steps
with $\alpha$ close to $\sigma$ are close to this choice.

\textbf{A factor of two.} For the unoriented model, the upper bound is twice the lower bound \eqref{eq:schulman}
as $\sigma\to0$, and the numerical values~\cite{gori_one-dimensional_2017} indicate that it is the lower bound which is sharp in
this limit. The factor of two is the price of monotone paths, which use the bonds to one side only. Closing the gap
would require paths with steps in both directions. The sites visited by two such paths no longer form a renewal
process, and we leave the proof that $2\zeta(1+\sigma)\,C_c(\sigma)\to1$ as an open problem. We note that the known
results on the spread-out limit of percolation~\cite{penrose_spread-out_1993,bollobas_spreadout_2007}, in which the mean degree at the
threshold tends to one, concern bonds of a finite range in $d\ge2$ dimensions and do not cover this case.

\textbf{Larger $\sigma$.} As $\sigma$ grows the bounds move away from the numerical thresholds, and for
$\sigma>0.65$ they are close to the trivial bound $C_c\le1$. This is a limitation of the method: the window
\eqref{eq:window} closes as $\sigma\to1$. The transition at $\sigma=1$ is of a different nature~\cite{aizenman_discontinuity_1986},
and we do not expect path-counting arguments to be sharp near it.

\textbf{The percolation probability.} Theorems~\ref{thm:main} and~\ref{thm:block} bound the percolation probability
itself, not only the threshold. For example, at $\sigma=0.5$ and $C=0.8$, theorem~\ref{thm:main} with Sibuya steps
gives $\theta\ge\thor\ge0.39$.

\textbf{Interdependent networks.} In a multiplex of two LR networks the threshold is higher than
$C_c$~\cite{sharabi2026percolation}, and only lower bounds on it are known. It would be interesting to extend the present
argument to that case, where a path must be open in both layers.

\section*{Acknowledgments}

\textbf{AI disclosure statement.} During the preparation of this work the author used Claude Opus~5.5 (Anthropic)
to help with the derivation of the proofs and with the codes that evaluate the bounds, to search the literature, and
to draft the manuscript. The author checked every step of the proofs, verified every cited source against the
original, and reviewed all other AI-assisted outputs, and takes full responsibility for the content of the
manuscript.

\section*{Data availability statement}

No new experimental data were created in this study. The codes that evaluate the bounds and produce
figure~\ref{fig:bounds} are openly available at \url{https://github.com/guy531/long_range_upper_bound}.

\providecommand{\newblock}{}

\end{document}